\documentclass[letterpaper,12pt]{amsart}

\usepackage{enumerate}
\usepackage{mathtools}
\usepackage[abbrev]{amsrefs}
\usepackage{amssymb,amsmath,amsthm,xfrac}
\usepackage{xcolor,amscd}
 \usepackage{slashed}

\definecolor{darkgreen}{rgb}{0,0.5,0}
\definecolor{darkred}{rgb}{0.7,0,0}
\definecolor{darkblue}{rgb}{0,.2,.7}

\usepackage[colorlinks,
citecolor=darkred, linkcolor=darkgreen, urlcolor=darkblue,
 pdfpagelabels=true,
 unicode=true,
]{hyperref}

\hypersetup{
 pdfauthor={Nelia Charalambous and Nadine Große},
 pdftitle={On the Lp spectrum of the Dirac operator},
 pdflang={en}
}

\newtheorem{thm}{Theorem}[section]
\newtheorem{prop}[thm]{Proposition}

\newtheorem{Def}[thm]{Definition}
\newtheorem{lem}[thm]{Lemma}

\newtheorem{corl}[thm]{Corollary}
\newtheorem{example}[thm]{Example}

\renewcommand{\i}{\mathrm{i}}

\newcommand{\eps}{\varepsilon}

\newcommand{\define}{\mathrm{:=}}
\newcommand{\Dc}{\slashed{D}}

\def \<{\langle}
\def \>{\rangle}

\def \H{{\cal H}}

\def \H^0{{\cal H}^0 or}

\def \p{\partial}
\def \n{\nabla}
\def \beq{\begin{equation}}
\def \eeq{\end{equation}}

\def \n{\nabla}
\def \eref{\eqref}

\begin{document}



\title{On the $L^p$ Spectrum of the Dirac operator}

\author{Nelia Charalambous}
\address{Department of Mathematics and Statistics, University of Cyprus, Nicosia, 1678, Cyprus} \email[Nelia Charalambous]{charalambous.nelia@ucy.ac.cy}

 \author{Nadine Gro{\ss}e} \address{Mathematisches Institut, Universit\"at Freiburg, 79100 Freiburg, Germany} \email[Nadine Gro{\ss}e]{nadine.grosse@math.uni-freiburg.de}

\thanks{The first author is partially supported by a University of Cyprus Internal Grant. The second author is partially supported by DFG SPP2026.}
 %\date{\today}



\begin{abstract} Our main goal in the present paper is to expand the  known class of noncompact manifolds over which the $L^2$-spectrum of a general Dirac operator and its square is maximal. To achieve this, we first find sufficient conditions on the manifold so that the $L^p$-spectrum of the Dirac operator and its square is independent of $p$ for $p\geq 1$. Using the $L^1$-spectrum, which is simpler to compute, we generalize the class of manifolds over which the $L^p$-spectrum of the Dirac operator is the real line for all $p$. We also show that by applying the generalized Weyl criterion, we can find large classes of manifolds with asymptotically nonnegative Ricci curvature, or which are asymptotically flat, such that the $L^2$-spectrum of a general Dirac operator  and its square is maximal.
\end{abstract}
\maketitle

\section{Introduction}

 The spectrum of geometric operators like the Laplacian or the Dirac operator plays a fundamental role in several analytic and geometric problems. However, computing the spectrum is in general a difficult task without any additional assumptions on the symmetries and structure of the manifold. Depending on the manifold there are several methods at hand. For example, it is in general easier to compute the $L^1$-spectrum of an operator over a noncompact manifold (instead of its $L^2$-spectrum) due to the fact that $L^1$-estimates for approximate eigenfunctions are easier to obtain in comparison to $L^2$-estimates. This is also reflected in the fact that the $L^2$-spectrum is in general a subset of the $L^1$-spectrum. As a result, whenever there exists a $p$-independence result for the $L^p$-spectrum, one can determine the $L^2$-spectrum by computing instead the $L^1$-spectrum of the operator. Wang in \cite{Wang97}*{Thm.~3} was the first one to utilize such a technique to prove that the $L^2$-spectrum of the Laplacian acting on functions is $[0,\infty)$ (in other words maximal) on a much larger class of manifolds than was previously known;  he took advantage of the $p$-independence result of Sturm \cite{sturm}.\medskip

Our first main goal is to find sufficient conditions on the manifold such that a $p$-independence result holds for a general Dirac operator and its square for $p\geq 1$. Most of the $p$-independence results known up to now are for Laplace-type operators. For closed manifolds it is easy to see that the spectrum is
 independent of $p$. So the interesting cases are noncompact manifolds. Hempel and Voigt \cites{HemVo, HemVo2} proved $p$-independence of the spectrum for Schr\"odinger operators  with potentials admitting certain singularities on $\mathbb R^n$.
 Later Kordyukov \cite{Kordu} obtained such a $p$-independence result for uniformly elliptic operators with uniformly bounded smooth coefficients on manifolds of bounded geometry having subexponential volume growth. Independently, Sturm
\cite{sturm} showed the independence of the $L^p$-spectrum for a class of uniformly elliptic operators in divergence form on manifolds with uniformly subexponential volume growth and Ricci curvature bounded from below. Both results include the Laplacian acting on functions. The condition on subexponential volume growth is an essential one since
 the $L^p$-spectrum of the Laplacian on the hyperbolic space does depend on $p$
\cite{Davies}*{Thm.~5.7.1} and is the convex hull of a $p$-dependent parabola in the complex plane. Later the first author proved in \cite{Char1}*{Prop.~9} that under the assumptions of the result by Sturm and whenever the curvature tensor on $k$-forms in bounded below, then the Hodge-Laplacian acting on $k$-forms has a $p$-independent spectrum.  The case of the drifting Laplacian over a weighted manifold was studied in \cite{ChLu1}. \medskip

For the $p$-independence of the Dirac operator much less is known. The classical Dirac operator on hyperbolic spaces also has $p$-dependent spectrum as was calculated in \cite{AG16}. But a $p$-independence result for the spectrum of the Dirac operator and its square on nice enough noncompact manifolds was not previously known. In this paper we will prove a $p$-independence result for the square Dirac operator on $L^p$, which extends to the Dirac operator on $L^p$ whenever the manifold has the right dimension or certain symmetric structures. In particular, we will prove the following result.

\begin{thm} \label{thmD2Lp}
Let $M^n$ be a complete Riemannian manifold with Clifford bundle $S$ and associated Dirac operator $D$. Suppose that the Ricci curvature of $M$ and the Clifford contraction $\mathcal{R}$ in \eqref{eq:Wei} are both bounded below, and the volume of $M$ grows uniformly subexponentially. Then the $L^p$-spectrum of the square Dirac operator acting on $L^p$, $D_p^2$, is independent of $p$  for all $p\in [1,\infty]$. Moreover, the set of isolated eigenvalues of finite multiplicity and their algebraic multiplicity is independent of $p$  for all $p\in[1,\infty]$.\medskip

If additionally $\sigma(D_2)$ is symmetric and $\sigma(D_1)\neq \mathbb C$, then the $L^p$-spectrum of $D_p$ is independent of $p$  for all $p\in[1,\infty)$.
\end{thm}

Here the lower index on the operator refers to the $L^p$-spaces the operator acts on, see Section~\ref{sec2} for the precise definitions.\medskip

We make the assumption $\sigma(D_1)\neq \mathbb C$ so that we can compare the $L^p$-spectrum of $D_p$ with that of $D_p^2$ (see Section~\ref{sec5} for further details). We believe  that this assumption is not necessary in general or it can be replaced by more geometric assumptions, but we hope to address this issue in future work. For the particular case of the classical Dirac operator we will show:

\begin{corl}\label{corDLp}
 Let $M^n$ be a spin Riemannian manifold with classical Dirac operator $\Dc$. Suppose that the Ricci curvature of $M$ is bounded below, and the volume of $M$ grows uniformly subexponentially. Then the $L^p$-spectrum of $\Dc_p^2$ is independent of $p$  for all $p\in[1,\infty]$. \medskip

If additionally $\sigma (\Dc_1)\neq \mathbb{C}$ and if  $n\neq 3$ modulo $4$ or  $M$ has an orientation-reversing isometry that lifts to spin structure, then the $L^p$-spectrum of $\Dc_p$ is independent of $p$  for all $p\in[1,\infty)$.
\end{corl}

Part of the machinery used for the $p$-independence result by Sturm or the first author are upper estimates for the heat kernel as in \cite{saco}.  In order to prove Theorem~\ref{thmD2Lp} we adapt those semigroup methods in order to be able to apply them to the square of the Dirac operator $D^2$. One difference is that the semigroup to $D^2$ is no longer a contraction. However, under our curvature assumption it will still be dominated by the semigroup associated to the Laplacian plus a certain constant. This allows us to use techniques developed in \cite{HSU} and make the semigroup method work. The $L^p$-independence of the Dirac operator is also of significance. We prove that it can occur over manifolds with certain symmetries, and relies on proving that the generator of the semigroup coincides with the square of the Dirac operator on $L^p$. The proofs of Theorem \ref{thmD2Lp} and Corollary \ref{corDLp} are found in Sections \ref{sec2}-\ref{sec5}. \medskip

Our second main goal is to find large classes of manifolds where the $L^2$-spectrum of the Dirac operator and its square is maximal, but also where the $L^p$-spectrum is contained in the real line. Previously, not many examples where the $L^2$ spectrum of a Dirac operator is maximal were known, and they were mainly in settings with very particular requirements on the metric, e.g. by Kawai in \cites{Kaw1,Kaw2} and B\"ar in \cite{Barhyp}, and only for the classical Dirac operator. Kawai considered spin manifolds with an end, whose metric can be written in radial polar coordinates outside a compact set  such that  the sectional curvature decays in absolute value at a rate $cr^{-2},$ where $r$ is the radial distance function to a fixed point. By constructing approximate eigenspinors he proved that $\sigma(\Dc_2)=\mathbb{R}$ on these manifolds.   B\"ar considered hyperbolic manifolds with finite volume.  With our the $L^p$ independence result, but also via the generalized Weyl criterion of the first author and Lu, we are able to expand the collection of manifolds where one can prove that the  $L^2$-spectrum of a general Dirac operator and its square is maximal, and coincides with the $L^p$-spectrum for all $p\in [1,\infty)$. \medskip

In Section \ref{S6} we find two very general classes of manifolds where the $L^p$-spectrum of the square Dirac operator is the nonnegative real line for all $p\in[1,\infty]$, by using the $L^p$-independence result of the spectrum and computing instead the $L^1$-spectrum. One example treats the classical Dirac operator over manifolds with an end in a special warped product form, and the other manifolds with asymptotically nonnegative Ricci curvature that admit asymptotically $D^2$-harmonic spinors. Using Theorem \ref{thmD2Lp} we are also able to find large classes of manifolds where the $L^2$ spectrum of the Dirac operator is maximal, in other words it is $\mathbb{R}$, but also where $\sigma(D_p)=\mathbb{R}$ for all $p\in[1,\infty)$. Our examples provide broad classes of manifolds where the $L^2$-spectrum of Dirac operators and their square is maximal. \medskip

In Section \ref{S7} we use the generalized Weyl criterion by the first author and Lu \cite{ChLu2} to find another large class of manifolds where the $L^2$ spectrum of the Dirac operator and its square is maximal. We consider manifolds with asymptotically nonnegative Ricci curvature in the radial direction, which admit an asymptotically $D^2$-harmonic spinor. We also consider asymptotically flat manifolds in the sense of Bartnik, since they admit a Witten spinor that can be used to produce test eigenspinors for the classical square Dirac operator. Finally in Appendix \ref{sec_def_D} we prove the technical result that the generator of the heat semigroup of the Dirac operator on $L^p$ coincides with the square of the Dirac operator acting on $L^p$.

\section{Preliminaries} \label{sec2}

\subsection{Spinorial Preliminaries}

Let $(M^n,g)$ be a complete Riemannian manifold $(M^n,g)$. Let $S\to M$ be a \emph{Clifford bundle}, in other words it is a bundle of Clifford modules over $M$ and equipped with a fiberwise Hermitian metric $\<.,.\>$, a metric connection $\nabla$ on $S$ such that the Clifford action is skew-adjoint with respect to $\<.,.\>$, and $\nabla$ is compatible with the Levi-Civita connection on $M$, cf. \cite{Roe}*{Def.~3.4}. The sections of $S$ will be called \emph{spinors}.\medskip

The Riemannian metric induces an $L^2$-inner product on spinors over $M^n$  by \[( \varphi, \eta ) \define \int_M \<\varphi, \eta\>dv\] where $dv$ is the Riemannian measure.\medskip

Let $D_c$ be the associated Dirac operator defined on compactly supported smooth spinors $C^\infty_c(S)$, see \cite{Roe}*{Chap.~3}.  When $M$ is spin and $S$ is the associated spinor bundle, we will refer to the associated Dirac operator as the \emph{classical Dirac operator} denoted by $\Dc$. Note that if $M$ is spin, we always assume that a spin structure is fixed and already chosen. Many of the interesting properties that Dirac operators exhibit  follow from the study of the closely related square Dirac operator, $D^2$. By \cite{Roe}*{(3.8)}  there exists an operator  $\mathcal{R}\in \mathrm{End}(S)$ such that the square Dirac operator satisfies the following Weitzenb\"ock formula
\begin{align}\label{eq:Wei}D^2=\nabla^*\nabla +\mathcal{R}.
\end{align}
The operator $\mathcal{R}$ is known as the \emph{Clifford contraction} and acts as a tensor on spinors. This decomposition allows us to consider the square Dirac operator as a Schr\"odinger type operator, and hence obtain useful heat kernel bounds depending on the behavior of the potential-like term $\mathcal{R}$. For the classical Dirac operator $\Dc$ the above  Weitzenb\"ock formula is also referred to as a Schr\"odinger-Lichnerowicz formula.

\begin{example}
\begin{enumerate}[(i)]
\item In the case that $M$ is spin, $S$ is the associated spinor bundle and $\Dc$ is the classical Dirac operator, then $\mathcal{R}=\mathrm{scal}/4$ where $\mathrm{scal}$ is the scalar curvature of $M^n$. Recall that the scalar curvature of $M^n$ is the trace of the Ricci tensor.
\item The complexified exterior bundle of an oriented Riemannian manifold equipped with its natural metric and connection admits the structure of a Clifford bundle, see \cite{Roe}*{Ex.~3.19}. In this setting the associated Dirac operator is the de Rham operator  $d^*+d$. Then $\mathcal{R}$ is determined by the Riemannian curvature tensor of the underlying manifold, which in the  case of one-forms simplifies to the Ricci curvature tensor \cite{BGV}*{p.~130}.
\item  For complex manifolds $M$ the exterior bundle of the complexified tangent bundle carries the structure of a Clifford bundle, see \cite{BGV}*{p.~135}. In this setting smooth spinors are the smooth differential forms on $M$, and the associated Dirac operator is  the Dolbeaut operator $\sqrt{2}(\bar{\partial}+\bar{\partial}^*)$ plus an endomorphism. In the case of
K\"ahler manifolds this endomorphism vanishes and then the Dirac operator coincides with the Dolbeaut operator.
\end{enumerate}
\end{example}

\subsection{Spectral preliminaries}

For an operator $H$ on a Banach space $\mathcal{B}$, we say that a point $\lambda \in \mathbb{C}$ is a point in the \emph{resolvent set} of $H$, $\rho(H)$,  if $H- \lambda I$ has a bounded inverse on  $\mathcal{B}$. If  $\lambda \in \rho(H)$, then $(H- \lambda I)^{-1}$ is the \emph{resolvent operator}. The complement of the resolvent set in $\mathbb{C}$ is the \emph{spectrum} of the operator, $\sigma(H)$. In the general case the domain of an operator $H$ will be denoted $\text{dom}(H)$.

\begin{Def}
For $p\in [1,\infty)$ we denote by $D_p$ the completion of $D_c\colon C_c^\infty(S)\subset L^p(S)\to C_c^\infty(S)\subset L^p(S)$ with respect to the graph norm as an operator from $L^p$ to $L^p$. To simplify notation, we will often write $D_p\colon L^p\to L^p$.
\end{Def}

We will also refer to the Dirac operator as $D$ when the domain and context are clear.

\section{Semigroup Properties}

In this section we collect  properties of the $L^p$-semigroup associated to a general square Dirac operator $D^2$. Throughout this section $S$ is a Clifford bundle over a complete Riemannian manifold $(M,g)$ and $D$ is the associated Dirac operator. From now on we denote the bound of an operator $A\colon L^p \to L^q$ by $\|A\|_{p \to q}$.\medskip

Denote by $\Delta_2$ the nonnegative definite Laplacian on the space of $L^2$-integrable functions. The $L^p$-independence of the spectrum of $D^2_{{p}}$ relies on specific estimates for its heat kernel. The following generalized Kato inequality between $D^2$ and the operator $\Delta -K_1$ over $C_c^\infty(S)$, whenever the Clifford contraction is bounded below, is a key ingredient for eventually proving the above estimates.

\begin{lem} \label{Dom1} Let $D$ be a Dirac operator over a Clifford bundle $S$. Suppose that  the Clifford contraction  $\mathcal{R}$ in \eref{eq:Wei} satisfies $\mathcal{R}\geq - K_1$ for some $K_1\geq 0$. Then for any smooth compactly supported section of the Clifford bundle $\varphi \in C_c^\infty(S)$ we have the pointwise estimate
 \begin{align*} {\rm Re} \< D^2 \varphi, \varphi \>  \geq |\varphi|(\Delta  -  K_1) |\varphi|.
 \end{align*}
\end{lem}

\begin{proof}
 By the  Weitzenb\"ock  formula we have $D^2 = \nabla^*\nabla +\mathcal{R}$.
 Using
 \begin{align*}
  \< \nabla^*\nabla \varphi, \varphi\> =&\frac{1}{2} \Delta |\varphi|^2 + |\nabla \varphi|^2
  =|\varphi|\Delta |\varphi| -|d|\varphi||^2 + |\nabla \varphi|^2 \end{align*}
and the Kato inequality $|d|\varphi||\leq |\nabla \varphi|$ , we obtain
 \begin{align*}
 \rm{Re} \< D^2 \varphi, \varphi \> &=\rm{Re} \<  \nabla^*\nabla  \varphi, \varphi \>  +  \< \mathcal{R}(\varphi),\varphi\>\\
 & = |\varphi|\Delta |\varphi| -|d|\varphi||^2 + |\nabla \varphi|^2 +  \<\mathcal{R}(\varphi),\varphi\>
 \geq  |\varphi| (\Delta  - K_1) |\varphi|
 \end{align*}
given our lower bound on the Clifford contraction $\mathcal{R}$.
\end{proof}

We  denote by $D_2^2$ the unique self-adjoint extension of the square Dirac operator on  $L^2(S)$ which is given by the Friedrichs extension.  Lemma~\ref{Dom1} together with the
 same calculation as in the proof of \cite{HSU}*{Thm.~3.3} shows that $D_2^2$ and $\Delta_2 - K_1$
 satisfy Kato's inequality in the sense of condition $(e)$ of Theorem 2.15 in \cite{HSU}, where
 in our case $H=D_2^2,$ $\ \mathcal{H}=L^2(S)$, $\ K= \Delta_2-K_1$, $\ \mathcal{K}=L^2(M)$ and
$\ \mathcal {K}^+=\overline{C_c^\infty(M,\mathbb R_{\geq 0})}^{L^2}$. The following result is
 then immediate from $(a)$ in the same theorem.

\begin{prop} \label{Dom2}
Suppose that  $\mathcal{R}\geq -  K_1$ for some $K_1\geq 0$. Then the semigroup $e^{-t (\Delta_2 -K_1)}$  on $L^2(M)$ dominates the semigroup $e^{-tD_2^2}$,  in other words for any $\varphi \in \mathrm{dom}(D_2^2)\subset L^2(S)$
\[
\left|e^{-tD_2^2} \varphi \right| \leq  e^{-t (\Delta_2 -K_1)} \left| \varphi \right|.
\]
\end{prop}

Proposition~\ref{Dom2} is essential for the remaining results of this section. First of all, given that we have a complete Riemannian manifold with Ricci curvature bounded below, the heat operator $e^{-t \Delta}$ is a Markov semigroup which is a contraction on $L^p$ for all $1\leq p \leq \infty$ \cite{Dav2}.  Combining this fact with Proposition~\ref{Dom2} we immediately get the following lemma.
\begin{lem} \label{Dom3}
Under the assumptions of Proposition \ref{Dom2} the heat operator \[e^{-t D^2}\colon L^2(S)\cap L^p(S)\to L^2(S)\cap L^p(S)\]  can be extended to a bounded operator on $L^p\to L^p$ for all $p\in [1,\infty)$. In the case $p=\infty$, $e^{-t D^2}$ is defined as the dual of the operator on $L^1$. Even though it is not necessarily a contraction on $L^p$,  for all $t\geq 0$ and $p\in [1,\infty]$ we have the upper bound
\[
\|e^{-t D^2}\|_{p\to  p} \leq  e^{K_1 t}.
\]
\end{lem}

We would like to remark that this result is a consequence of Kato's inequality for $D^2$ which essentially implies that whenever $\mathcal{R}\geq -  K_1$,
\[
\int_M \<d |\varphi|, d |\varphi| \> \,  dv\leq \int_M \< D \varphi, D \varphi \> \,  dv +   K_1 \int_M  |\varphi|^2  dv.
\]
for all $\varphi\in C_c^\infty(S)$. For the Laplacian on functions, one can prove Lemma \ref{Dom3} with $K_1=0$ in a more direct way,  using  Davies~\cite{Davies}*{Theorems 1.3.2, 1.3.3} because it is a real-valued operator and satisfies the pointwise Kato inequality $|\, d|u|\, | \leq | \nabla u | (=| d u |)$. \medskip

Whenever the Clifford contraction is bounded below, Lemma \ref{Dom3} allows us to define the infinitesimal generator, $H_p$, of the semigroup $e^{-tD^2}$ on $L^p$ for all $1\leq p\leq\infty$. In Appendix~\ref{sec_def_D} we prove that for $1\leq p<\infty$ the infinitesimal generator $H_p$ coincides---as one expects---with   the square of the Dirac operator $D_p^2$ on $L^p$. Note that it is still unclear whether $H_{\infty}$ coincides with $(D_{1}^*)^2$ in general. To simplify notation we will denote $H_p$ by $D_p^2$ for all $1\leq p\leq\infty$ from now on.  \medskip

Proposition~\ref{Dom2} also allows us to prove the existence of an integral kernel for the operator $e^{-t D_2^2}$ and obtain a Gaussian-type upper bound for it, whenever such an upper bound holds for the Laplacian on functions. We can prove the following.

\begin{prop} \label{Tb1}
Consider a complete Riemannian manifold $M^n$ with Ricci curvature bounded below $\rm{Ric}\geq -K_0$ for some $K_0\geq 0$. Let $S\to M$ be a Clifford bundle with a Dirac operator $D$ that satisfies \eref{eq:Wei} and assume that the Clifford contraction is bounded below, namely $\mathcal{R}\geq - K_1$ for some $K_1\geq 0$.  Then  the heat kernel, $\vec{H}(x,y,t)$, corresponding to the square Dirac operator $D_2^2$ exists, and for any  $x,y\in M$, $t>0$ and $\delta\in(0,1)$  it satisfies the upper bound

\begin{equation}
\begin{split}\label{heatg}
\left|\vec{H}(x,y,t) \right| \leq  \;C_3(\delta, n)  & \;   V^{-1/2}(x,\sqrt{t}) \,V^{-1/2}(y,\sqrt{t}) \\
&   \cdot  \exp \left[ -\, \frac{d^2(x,y)}{C_4(\delta, n)\,t} + C_5(n) \sqrt{K_0\,t}  + K_1 t \,\right]
\end{split}
\end{equation}
for some positive constants $C_3(\delta,n),  C_4(\delta,n)$ and  $C_5(n).$
\end{prop}
In the above estimate $V(x,r)$ denotes the Riemannian volume of  the geodesic ball of radius $r$ centered at $x$, and $d(x,y)$ is the distance between $x, y$ induced by the Riemannian metric.

\begin{proof}
Our lower bound on Ricci curvature implies that there exists a unique heat kernel $h(x,y,t)$ for the Laplacian on functions, \cite{Chavel}*{p. 188/189} with a Gaussian type upper bound given in \cite{saco}*{Theorem 6.1}.  The pointwise domination property of Proposition~\ref{Dom2} also implies the existence of a heat kernel for  $D_2^2$, $\vec{H}(x,y,t)$. By following the proof of Theorem 3.5 in \cite{Ros}, this domination property can be used to show
\[
\left|\vec{H}(x,y,t)\right| \leq \left| h(x,y,t)\,  e^{K_1 t}    \right|.\qedhere
\]
\end{proof}

We  refer the interested reader to the book \cite{BGV}*{Chapter 2} for a formal definition of the heat semigroup for operators such as $D_2^2$ over Clifford bundles, as well as a detailed construction of its heat kernel in the case the manifold is compact. \medskip

Note that through the same argument one also has the following.
\begin{corl}

Suppose that $\Dc$ is the classical Dirac operator over a complete Riemannian spin manifold. Suppose that
 the scalar  curvature of $M$ is bounded below by  $\mathrm{scal}\geq - 4K_1$ for some $K_1 \geq 0$. Then the self-adjoint extension of $\Dc^2$ on $L^2(S)$ satisfies the domination property of Proposition \ref{Dom2}  and the upper bound of Lemma \ref{Dom3} for its semigroup. Moreover, whenever the Ricci curvature of $M$ is bounded below $\mathrm{Ric}\geq-K_0$ for some $K_0 \geq 0$, then the heat kernel of the semigroup $e^{-t \Dc_2^2}$ exists and satisfies the Gaussian upper bound of Theorem \ref{Tb1}  with $K_1=\frac{n}{4} K_0$.
\end{corl}

\section{\texorpdfstring{$L^p$}{Lp}-Independence of the Spectrum of  \texorpdfstring{$D_p^2$}{D2}}

In this section we prove the main part of Theorem \ref{thmD2Lp} -- the $p$-independence of the spectrum of $D^2_p$. We will prove the second part of the theorem concerning the spectrum of $D_p$ in the next section.  \medskip

We begin with some preliminary definitions and notations. Recall that $\rho(D^2_p)$ denotes the resolvent set of $D^2_p$, and  $\sigma(D^2_p)=\mathbb{C}\setminus \rho(D^2_p)$ is its spectrum. We will show that under the assumptions of Theorem~\ref{thmD2Lp} the resolvent set is  independent of $p$. For this, it suffices to show that $\rho(D^2_p)\subset \rho(D^2_2)$ and $\rho(D^2_2)\subset \rho(D^2_p)$ for all $p$. The first containment is the one that holds under more restrictive geometric assumptions, whereas the other one holds in a more general setting as we will see.\medskip

We note that in general $\sigma(D^2_p) =   \sigma(D^2_q)$ whenever $1/p + 1/q=1$ by duality even in the case $p=1$,  since $ D^2_\infty  =    H_\infty$ is by definition the dual operator to $ D^2_1$.\medskip

On $M^n$ we denote as $B(x,r)$ the geodesic ball of radius $r$ centered at $x$, and the Riemannian volume of $B(x,r)$  is again denoted as $V(x,r)$. The following property is key when proving an $L^p$-independence result for the spectrum.

\begin{Def} \label{unisubexp}
A Riemannian manifold is said to have \emph{uniformly subexponential volume growth}, if for any $\eps>0$ there exists a constant $C{(\eps)}$, independent of $x$, such that for any $r>0$ and $x\in M$
\[
V(x,r)\leq C{(\eps)} \, V(x,1) \, e^{\eps \, r }.
\]
\end{Def}

Define $\mu(x)\define V(x,1)^{-1/2}$. The assumption on the uniformly subexponential volume growth implies in particular the following estimate:
\begin{lem}{\cite{sturm}*{Proposition 1}} \label{VolCom2}
Let $M^n$ be a Riemannian manifold whose volume grows uniformly subexponentially. Then for any $\beta >0$
\[
\sup_{x\in M}  \int_M \mu(x) \, \mu(y) \, e^{-\beta \, d(x,y)} \, dv(y) < \infty,
\]
where $d$ is the distance function on $(M,g)$.
\end{lem}

To prove the $L^p$-independence of the spectrum of $D_p^2$, we will use the heat kernel estimate technique that was introduced by \cite{sturm}, although we need to be a bit more careful with our arguments as our operator is defined on spinors. A set of auxiliary results is necessary, which we outline below. \medskip

We begin by defining the set of functions
\[
\Psi_{\eps} \define \{ \psi \in {\mathcal C}_0^{1}(M) \, | \, |\n \psi|\leq \eps  \}
\]
where ${\mathcal C}_0^{1}(M)$ are compactly supported $C^1$-functions on $M$. The set $\Psi_\eps$ will be used to define  perturbations of the operator. Note that for any pair of points $x, y \in M$, the Lipschitz property of $\psi\in \Psi_\epsilon$ implies
\begin{align*}\sup\{|\psi(x)-\psi(y)|\, | \, \psi \in  \Psi_{\eps}\} \leq \eps \, d(x,y).\end{align*}

We consider perturbations of the operator $D^2_2$ of the form $e^{\psi} D^2_2 e^{-\psi}$. The next lemma is an analytic result that is independent of the geometric conditions on the manifold.

\begin{lem} \label{lemkato}
For any compact subset $K$ of the resolvent set of $\rho(D^2_2)$ there exist $\eps>0$ and $C<\infty$ such for all $\xi \in K$ and $\psi\in \Psi_{\eps}$, $\xi$ belongs to the resolvent set of the operator $e^{\psi}D^2_2 e^{-\psi}$ and
\[
\|(e^{\psi} D^2_2 e^{-\psi}-\xi)^{-1}\|_{2\to 2} \leq C.
\]
\end{lem}

\begin{proof}
The proof follows as in \cite{sturm}*{Lem.~1} since  the computation does not change when taking sections in a Clifford bundle and $H_2=D_2^2$. In particular, the proof uses Theorem VI.3.9 in \cite{Kato} to get a uniform bound for the resolvent operator of these perturbations over compact sets of the resolvent set of $D^2_2$.
\end{proof}

Our next goal is to prove the $L^p$-boundedness of some power of the resolvent operator, $(D^2_2-\xi)^{-m}$,  whenever $\xi$ belongs to $\rho(D^2_2)$. This requires the geometric assumptions on the manifold as it relies on upper Gaussian estimates for the heat kernel of $D^2_2$ and volume comparison estimates. These lead to the following strong result about the heat kernel.

\begin{lem} \label{Lem1}
Consider a complete Riemannian manifold $M^n$ and    a Dirac operator $D$ on a Clifford bundle over $M$  which satisfies \eref{eq:Wei}. Suppose that the Clifford contraction and Ricci curvature are bounded below such that $\mathcal{R} \geq - K_1$  and $\mathrm{Ric}\geq -K_0$, for some $K_1, K_0\geq 0$. Then for any $\beta >0$ there exists a negative number $\alpha$ and constant $C(n,K_0,\beta)$ such that the heat kernel of $D^2_2$ satisfies the upper bound
\begin{equation} \label{heat42}
\left|\vec{H}(x,y,t) \right|   \leq  C \, \mu(x)^2\, \sup\{t^{-n/2},1\} \,e^{-\beta \, d(x,y)}\,
 e^{- (\alpha +1) t}.
\end{equation}
\end{lem}

\begin{proof}
The proof of the lemma uses the Gaussian upper bound that we proved in Proposition \ref{Tb1}, and the assumption on Ricci curvature (see for example \cites{Char1,sturm,ChLu1} for similar estimates). \medskip

By the Bishop volume comparison Theorem, whenever $\rm{Ric}\geq -K_0$, there exists $C=C(n)$ such that
\begin{equation}\label{VolCom}
\frac{V(x,R)}{V(x,r)} \leq \,\frac{R^n}{r^{n}} \, e^{C\,\sqrt{K_0}\, R}
\end{equation}
for any $0<r\leq R$. \medskip

Taking $r<1=R$, we have
\[
V(x,r)^{-1} \leq \mu(x)^2 \, r^{-n} \, C(n, K_0),
\]
whereas
\[
V(x,r)^{-1} \leq \mu(x)^2
\]
holds for any $r\geq1$ by monotonicity of the volume. Therefore, for any $r>0$
\begin{equation}\label{VolCom4}
V(x,r)^{-1} \leq \mu(x)^2 \sup\{ r^{-n},1\} \, C(n, K_0),
\end{equation}
 The volume comparison \eref{VolCom} also yields
\begin{equation}\label{VolCom3}
\mu(y)^2 \leq C\,\mu(x)^2\,  e^{\bar C\, (\sqrt{K_0} +1) \, d(x,y)}
\end{equation}
for all $x, y \in M$ and constants $C,\bar C$ that only depend on $n$. \medskip

Moreover, for any $\gamma \in \mathbb{R}$, the exponential function satisfies the inequality
\begin{equation}\label{VolCom5}
e^{-d^2/4ct} \leq e^{-\gamma \, d} \, e^{c\, \gamma^2 t}.
\end{equation}

Using \eref{VolCom4}, \eref{VolCom3} and \eref{VolCom5} we can estimate the right side of \eref{heatg} to get
\begin{equation*}
\begin{split}
\left| \vec{H}(x,y,t) \right| &\leq \;  C \, \mu(x)^2\, \sup\{t^{-n/2},1\} \; \\
&\cdot \, \exp\left[  \left(\bar C\,( \sqrt{K_0} + 1)  - \gamma \right) d(x,y)  + C_5 \sqrt{K_0 \,t} +K_1\, t  + \frac 14 \gamma^2 C_4 \,t \,\right].
\end{split}
\end{equation*}
Choosing $\gamma$ such that $\beta= - \bar C\, (\sqrt{K_0} +1) + \gamma$ is any positive number, we get the desired estimate with $\alpha=-1 -C_5 \sqrt{K_0} - K_1 - \frac 14 \gamma^2 C_4 .$
\end{proof}

We are now ready to state and prove the $L^p$-boundedness of the resolvent operator. The proof is a standard analytical argument which we include for completion.
\begin{prop}   \label{Prop1}
Consider a complete Riemannian manifold $M^n$ and $D$   a Dirac operator on a Clifford bundle over $M$ which satisfies \eref{eq:Wei}. Suppose that the Clifford contraction and Ricci curvature are bounded below such that $\mathcal{R} \geq - K_1$  and $\rm{Ric}\geq -K_0$, for some $K_1, K_0\geq 0$.  Assume moreover that the volume of $M$ grows uniformly subexponentially.

Let $K$ be a compact subset of $\rho(D^2_2)$. Then for any $\xi\in K$,  and for any $m > n+2$,  $(D^2_2 -\xi)^{-m}$ has a smooth kernel function $\vec{G}_\xi(x,y)$ satisfying
\begin{equation} \label{res1}
|\vec{G}_\xi(x,y)|\leq C\mu(x)\mu(y) e^{-\eps d(x,y)}
\end{equation}
for some $\eps>0$ and $C>1$.

In addition, the operator $(D^2_2-\xi)^{-m}|_{L^2\cap L^p}$ extends to a bounded operator from $L^p$ to $L^p$ for all $1\leq p \leq \infty$.
\end{prop}

Dunford and Pettis proved that an operator on functions which is bounded from $L^1$ to $L^\infty$  has an integral kernel. In \cite{Char1}*{Lem.~11} the first author showed that an integral kernel exists for an operator on the $k$-form bundle $\Lambda^k(M)$ which is bounded from $L^1$ to $L^\infty$. By following the same argument, but for a Clifford bundle, we can also show that an operator has an integral kernel whenever it is bounded from $L^1$ to $L^\infty$:

\begin{proof}
 In order to prove \eqref{res1} it suffices to show that the perturbed resolvent operator $\mu^{-1} \, e^{\psi} (D_2^2 -\xi)^{-m} e^{-\psi} \mu^{-1}$ is bounded from $L^1$ to $L^\infty$ for all $\psi\in \Psi_\eps$ for some $\eps>0$. To this end,  we use powers of the resolvent equation \cite{Kato}*{5.5 on p.~36} to get
\begin{equation} \label{e3_1}
\begin{split}
\mu^{-1} \, e^{\psi} (  D^2_2 -\xi)^{-m} e^{-\psi} \mu^{-1}
&=
\sum_{j=0}^m \left(
\begin{array}{c}
m \\
j \\
\end{array}
\right)
(\xi-\alpha)^j[\mu^{-1} \, e^{\psi} (D^2_2 -\alpha)^{-m/2} e^{-\psi}]\\
& \qquad \cdot [e^{\psi} (D^2_2 -\xi)^{-1} e^{-\psi}]^{j} \; [ e^{\psi} (D^2_2 -\alpha)^{-m/2} e^{-\psi}\mu^{-1}]
\end{split}
\end{equation}
where $\alpha$ is a negative constant. The resolvents on the right side of the equations commute. But this splitting is convenient in order to show that this actually is a bounded operator from $L^1$ to $L^\infty$, see below.\medskip

On the right side, we have in essence three operators. First note that by Lemma \ref{lemkato}, the middle operator is bounded on $L^2$. In other words there exists an $\eps,C >0$, depending on $\xi$, such that for all $\psi\in \Psi_\eps$
\begin{equation}\label{e3b}
\|e^{\psi} (D^2_2 -\xi)^{-1} e^{-\psi}\|_{2 \to 2} \leq C.
\end{equation}

Since $\alpha<0$, the operator $(D^2_2 -\alpha)^{-m/2}$ satisfies the identity (see \cite{EN}*{p.57})
\[
(D^2_2 -\alpha)^{-m/2} = C(m) \int_0^\infty e^{-t \, D^2_2} \, t^{\frac{m}{2}-1} \, e^{\alpha t} \,dt.
\]
By Proposition~\ref{Tb1} the heat operator has a kernel. Hence, the integrability of the right hand side of the last equation guarantees the existence of a kernel for the resolvent operator $(D^2_2 -\alpha)^{-m/2}$ as well. We denote this kernel by $\vec{G}^{m/2}_\alpha(x,y)$. By Lemma \ref{Lem1} we know that right side is integrable whenever $m>n+2$  and for $\alpha$ sufficiently negative.
More specifically, for any  $\beta>0$ and $m>n+2$ we can find an $\alpha$ sufficiently negative such that
\[
|\vec{G}^{m/2}_\alpha(x,y)| \leq C(n, K_0,\beta)\,\mu(x)^2 e^{-\beta\, d(x,y)}.
\]

For any  $\omega\in L^2$, if we use the estimate \eref{VolCom3}, Lemma \ref{VolCom2} and choose $\beta>0$ large enough, we can see that
\begin{equation*}
\begin{split}
\| \mu^{-1} \, e^{\psi} ( D_2^2 -\alpha)^{-m/2} &e^{-\psi} \omega \|_{L^\infty}
= \sup_{y} \mu(y)^{-1} \, e^{\psi(y)} \bigl| \int_M \vec{G}^{m/2}_\alpha(x,y) e^{-\psi(x)} \, \omega(x) \, dv(x) \bigr|\\
&\leq C \sup_{y} \bigl| \int_M \mu(y)^{ -2} \,  \mu(x)^2  e^{2(\eps-\beta)d(x,y)}  \, dv(x) \bigr|^{1/2}\, \|\omega(x)\|_{L^2} \\
&\leq C \|\omega(x)\|_{L^2}.
\end{split}
\end{equation*}

In other words,
\[
\|\mu^{-1} \, e^{\psi} (D^2_2 -\alpha)^{-m/2} e^{-\psi}\|_{2 \to \infty} \leq C.
\]
By taking adjoints, we   get that the third operator in the right side of \eref{e3_1} is bounded from $L^1$ to $L^2$,
\[
\|e^{\psi} (D^2_2 -\alpha)^{-m/2} e^{-\psi}\mu^{-1}\|_{1 \to 2} \leq C.
\]

Combing the three estimates for the operators on the right side of \eref{e3_1}, we get that  $\mu^{-1} \, e^{\psi} (  D^2_2 -\xi)^{-m} e^{-\psi} \mu^{-1}$ is bounded from $L^1$ to $L^\infty$. Hence it has an integral kernel. The pointwise bound on $ \vec{G}^{m}_\xi(x,y)$  follows. \medskip

Lemma \ref{VolCom2} now implies that  $(D^2_2 -\xi)^{-m}$ is a bounded operator on $L^1$. By the definition of $\xi$ it is also bounded  on  $L^2$. Using a standard  interpolation argument, it follows that it is bounded on $L^p$ for any $1\leq p\leq 2$. Taking adjoints we get that it is bounded for any  $1\leq p\leq \infty$.
\end{proof}

We are now ready to prove the $L^p$-independence  for $\sigma(D_p^2)$:

\begin{proof}[Proof of the $p$-independence of $\sigma(D^2_p)$ in  Theorem \ref{thmD2Lp}]

We will first prove the inclusion $\sigma(D^2_p)\subset \sigma(D^2_2)$ which is the most difficult one to demonstrate. \medskip

By Proposition  \ref{Prop1}, with $ m= n+3$, the operator $(D^2_2 -\xi)^{-m}$  is  bounded on $L^p$ for all $p\in[1,\infty]$ and for all $\xi \in \rho(D^2_2)$. Given that $\sigma(D^2_2)$ is contained in the nonnegative real line, any $\xi \in \rho(D^2_2)$ can be path connected to a negative number. By uniqueness of extension we conclude that $(D^2_p -\xi)^{-m}= (D^2_2 -\xi)^{-m}$ for all $\xi \in \rho(D^2_2)$ (see also \cites{HemVo,HemVo2,sturm,Char1,ChLu1} for further details). \medskip

Recall that the spectrum $D^2_p$ is closed for all $p$. If $\sigma(D^2_p)$ is not contained in $\sigma(D^2_2)$, then we can find a point $\xi_o$ on the boundary of $\sigma(D^2_p)$ which does not belong to $\sigma(D^2_2)$. Moreover, for any $\eps> 0 $, we can find $m-1$ distinct points $\xi_j \in \rho(D^2_p)$ for $1\leq j \leq m-1$  with $|\xi_j-\xi_0|<\epsilon$. Using the abstract resolvent equation \cite{Kato}*{Equation (5.6) p. 37}, if the right-hand side of
\begin{equation*}
\begin{split}
&\left((D^2_p-\xi_o)(D^2_p-\xi_1)\cdots (D^2_p-\xi_{m-1})\right)^{-1}\\
& \quad =\sum_{l_1,\cdots, l_{m-1}= 0}^{\infty} (D^2_p-\xi_o)^{-m-l_1-\cdots-l_{m-1}}(\xi_1-\xi_o)^{l_1}\cdots(\xi_{m-1}-\xi_o)^{l_{m-1}}.
\end{split}
\end{equation*}
is bounded on $L^p$, then the operator $(D^2_p-\xi_o)(D^2_p-\xi_1)\cdots (D^2_p-\xi_{m-1})$ has a well-defined inverse which is a bounded operator in $L^p$.\medskip

Recall that $(D^2_2-\xi_o)^{-m}$ is bounded on $L^p$ by Proposition \ref{Prop1}. Since $\|(D^2_p-\xi_o)^{-m}\|_{p \to p} = \|(D^2_2-\xi_o)^{-m}\|_{p \to p}=C$, then
\[
\|(D^2_p-\xi_o)^{-l}\|_{p \to p} \leq C^{l/m}
\]
for any $l\geq m$. Choosing $\eps$ small enough, the resolvent expansion above gives us that $\left((D^2_p-\xi_o)(D^2_p-\xi_1)\cdots (D^2_p-\xi_{m-1})\right)^{-1}$ is bounded on $L^p$. \medskip

At the same time, by \cite{Kato}*{(5.5) on p.~36}
\[
\left((D^2_p-\xi_o)(D^2_p-\xi_1)\cdots (D^2_p-\xi_{m-1})\right)^{-1}
  =\alpha_o(D^2_p-\xi_o)^{-1}+\sum_{j=1}^{m-1} \alpha_j(D^2_p-\xi_j)^{-1}
\]
where $\alpha_o \ldots \alpha_{m-1}$ are constants (depending on $\xi_j$).  As a result, $(D^2_p-\xi_o)^{-1}$ is bounded on $L^p$, since every other operator in the above equation is bounded on $L^p$. In other words, $\xi_o$ is in the resolvent set, which is a contradiction. \medskip

The other containment, $\sigma(D^2_p)\supset \sigma(D^2_2)$, holds in a very general setting and only requires an interpolation argument. The idea is to consider $D^2_p$ and $D^2_{p^*}$ where $p, p^*$ are dual such that $\frac 1p + \frac{1}{p^*}=1$. Then  $\rho(D^2_p)=\rho(D^2_{p^*})$ and the resolvent operators $(D^2_p-\xi)^{-1}$, $(D^2_{p^*}-\xi)^{-1}$ are  bounded  on $L^p$ and $L^{p^*}$ respectively for all $\xi\in \rho(D^2_p)$. Using the  Calder\'on Lions Interpolation Theorem  \cite{ReSiII}*{Theorem IX.20} and the argument by Hempel and Voigt  \cites{HemVo2}, we get that the interpolated bounded operator on $L^2$ is in fact $(D^2_2-\xi)^{-1}$. Therefore $\xi\in \rho(D^2_2)$.\medskip

The proof of the $L^p$-independence of isolated eigenvalues of finite multiplicity follows similarly as in \cite{HemVo}*{Proposition 3.1} (see also \cites{Aut,Char1}). The main idea is to consider the Laurent series expansion for the resolvent operator $(D^2_p-\xi)^{-1}$ on a small disc around the eigenvalue $\lambda$  of finite multiplicity, which contains no other points in the spectrum other than $\lambda$. The order of the pole at $\lambda$ corresponds to the multiplicity of the eigenvalue. The $L^p$-independence of the spectrum allows us to show that the Laurent series expansions are the same for all $p$, and in consequence the order of the pole is also the same.
\end{proof}

\section{The \texorpdfstring{$L^p$}{Lp} spectrum of \texorpdfstring{$D_p$}{D}}\label{sec5}

Up to now we examined when the $L^p$-spectrum of $D_p^2$ is independent on $p$. Of course, it is interesting to know whether the $L^p$-spectrum of $D_p$ itself is also independent of $p$.\medskip

First by the Riesz-Thorin interpolation theorem $\sigma (D_2)\subset \sigma (D_q)\subset \sigma (D_p)$ for all $\infty> p\geq q\geq 2$ or $1\leq p\leq q\leq 2$. Moreover, we note that the proof of Lemma~B.8  in \cite{AG16} stays exactly the same for Dirac operators on arbitrary Clifford bundles over complete Riemannian manifolds. Hence, we know that if the  spectrum $\sigma (D_p)$ is not the entire complex plane, then $\lambda^2 \in \sigma (D^2_p)$ if and only if  $\lambda$ or $-\lambda$ belongs to $ \sigma (D_p)$. This is the only statement where $\sigma (D_1)\neq \mathbb C$ is used. We think that for the manifolds considered in Theorem~\ref{thmD2Lp} this is actually automatically the case but up to now we were not able to prove this in full generality. We hope to address such questions in future work. If $(M,g)$ and the Clifford bundle are of bounded geometry, the assumption $\sigma(D_1)\neq \mathbb C$ is automatically fulfilled, see \cite{AG16}*{Rem.~B9} and cp. Lemma~\ref{app_b}.\medskip

\begin{proof}[Proof of Theorem~\ref{thmD2Lp} -- the last statement  concerning $\sigma(D_p)$]
As proved in the  last section $\sigma(D_p^2)$ is independent of $p$ -- in particular $\sigma (D_p^2)=\sigma (D_2^2)$. From the last paragraph above we obtain that whenever $\sigma(D_1)\neq \mathbb C$ then also $\sigma(D_p)\neq \mathbb C$ for $p<\infty$ and hence $\lambda^2 \in \sigma (D^2_p)$ if and only if  $\lambda$ or $-\lambda$ belongs to $ \sigma (D_p)$. Now consider a point $\lambda\in \sigma (D_p)$. Then, $\lambda^2\in \sigma (D_p^2)=\sigma (D_2^2)$. Since  $\sigma(D_2)$ is assumed to be symmetric, $\lambda\in \sigma (D_2)$. Hence, $\sigma (D_p)\subset \sigma (D_2)$. Together with the converse inclusion, we conclude that $\sigma (D_p)$ is independent of $p$.
\end{proof}

\begin{proof}[Proof of Corollary~\ref{corDLp}] For the special case of the classical Dirac operator  \cite{AG16}*{Lemmas~B.10 and~B.11} gives that for $\text{dim}\, M\neq 3$ modulo $4$, or if $M$ admits an orientation reversing isometry that lifts to the spin structure, the $L^p$-spectrum of $\Dc_p$ is always symmetric. Hence, we can use the last theorem which completes the proof of Corollary~\ref{corDLp}.
\end{proof}

\section{Computation of the spectrum of \texorpdfstring{$D_p$}{Dp} and \texorpdfstring{$D_p^2$}{Dp2} via computation of \texorpdfstring{$L^1$}{L1} spectrum.} \label{S6}

In this section, we give two examples (one very concrete and a more abstract one) where we prove that the $L^p$-spectrum of the square Dirac operator is the nonnegative real line for all $p\in[1,\infty)$. We achieve this by using the $p$-independence result of Theorem \ref{thmD2Lp}, and showing that the $L^1$- spectrum contains the nonnegative real line. This implies that $\sigma(D^2_2)$ is maximal in these cases. Using the same theorem, we can also find large classes of manifolds where $\sigma(D_2)$ is maximal (in other words equal to $\mathbb{R}$) and where $\sigma(D_p)=\mathbb{R}$ for all $p\in[1,\infty)$.\medskip

Our first example treats the classical Dirac operator over manifolds with ends in a special warped product form. The second example considers the classical Dirac operator over manifolds that admit a bounded, asymptotically harmonic spinor. In both cases the geometric conditions on the structure of the manifold are such that the $L^p$ spectrum of the square Dirac operator is $p$-independent.

\subsection{Some manifolds with warped product ends}
In this section we want to apply our results to  a class of manifolds with warped product ends.

\begin{thm} Let $(\tilde{M}, \tilde{g})$ be a complete spin Riemannian manifold of dimension $n$ which has uniformly subexponential volume growth such that there exists an isometric embedding $(M=N\times [0,\infty),g=f(t)^2g_N+dt^2)\hookrightarrow (\tilde{M}, \tilde{g})$ with $\text{ker}\,  \Dc^{N}\neq \varnothing$. Assume that  $f$ is positive, monotonically increasing with
\begin{equation} \label{grow1}
\lim_{t\to \infty} \tfrac{f^{n-1}(2t)}{tf^{n-1}(t)}\to 0.
\end{equation}

Then, the $L^p$-spectrum of the square of the  classical Dirac operator on $(\tilde{M}, \tilde{g})$ is $[0,\infty)$.

If additionally $\mathrm{dim}\, \tilde{M}\neq 3\ \text{mod}\ 4$ or $M$ admits an orientation reversing isometry that lifts to the spin structure, then, the $L^2$-spectrum of the  classical Dirac operator on $(\tilde{M}, \tilde{g})$ is $\mathbb R$, it is in other words maximal. If we also know that $\sigma(\Dc_1)\neq \mathbb{C},$ then  $\sigma(\Dc_p)=\mathbb{R}$ for all $p\in[1,\infty)$.
\end{thm}

Admissible functions $f$ include some functions with quasi-polynomial growth such as   $f^{n-1}(t)=e^{\frac{1}{4}(\ln t)^2}$. Also functions with polynomial growth are admissible in the above theorem. Note that for polynomial growth functions one could directly compute the $L^2$-spectrum, but not for the general class of $f$ that satisfies \eqref{grow1}.

\begin{proof}

By \cite{BGM}*{(3.6)} the Dirac operator on the end $M$ is given by
\[\Dc^M=\nu\cdot \left( \tilde{\Dc}^{N} +\frac{n-1}{2} \frac{f'(t)}{f(t)} +\nabla^M_{\nu} \right),\]
 where $\widetilde{\Dc}^{N}=\Dc^N$ when the dimension of $N$ is even, $\tilde{\Dc}^{N}=\text{diag}(\Dc^N, -\Dc^N)$ otherwise, and  $\nu\define \partial_t$.  Let $\phi\in \mathrm{ker}\, \tilde{\Dc}^{N}$ be such that $\nu\cdot \phi=\i \phi$. This is possible since $\nu^2=-\text{Id}$ and $\tilde{\Dc}^N(\nu \cdot .)=-\nu\cdot \tilde{\Dc}^N (.)$. \medskip

 Let $\eta_T\colon \mathbb R\to [0,1]$ be a smooth cut-off function with $\text{supp}\, \eta_T=[T,4T]$, $\eta_T|_{[2T,3T]}=1$ and $|\eta_T'|\leq 2/T$.
We define spinors $\psi_T$ with support in $M$ by setting  $\psi_T(x,t)= \eta_T(t) \text{exp}(- \lambda\i t)\phi(x,t)$ for some $\lambda\in \mathbb R$ and all $(x,t)\in M=N\times [0,\infty)$.
 Then
 \begin{align*}
  (\Dc^M-\lambda)\psi_T=& \left(\frac{n-1}{2}\frac{f'}{f} \eta_T + (\eta_T' -\lambda\i \eta_T)  \right)\text{exp}(- \lambda\i t)  \nu\cdot \phi - \lambda \eta_T \text{exp}(- \lambda\i t)  \phi\\
  =& \left(\frac{n-1}{2}\frac{f'}{f} \eta_T + \eta_T' \right)\text{exp}(- \lambda\i t)  \nu\cdot \phi
 \end{align*}
which implies
 \begin{align*}
&\frac{\Vert (\Dc^M -\lambda) \psi_T\Vert_{L^p}}{\Vert \psi_T\Vert_{L^p}} \leq
\frac{\Vert \frac{n-1}{2}\frac{f'}{f} \eta_T \phi\Vert_{L^p} + \Vert \eta_T' \phi\Vert_{L^p}}{\Vert \psi_T\Vert_{L^p}}\\
\leq &
\frac{\left( \int_{N\times [T, 4T]} \left( \frac{n-1}{2} \left|\frac{f'(t)}{f(t)}\right| \, |\phi|\right)^p f^{n-1}(t) \, dt \, dv_N \right)^{\frac{1}{p}} + \frac{2}{T} \left(\int_{N\times [T, 4T]} |\phi|^p \, f^{n-1}(t) dt \, dv_N \right)^{\frac{1}{p}}
}{\Vert \phi\Vert_{L^p(N)} \left(\int_{ [2T, 3T]} f^{n-1}(t) \, dt\right)^{\frac{1}{p}}}\\
\leq &
\frac{\left( \int_{ [T, 4T]} |\frac{n-1}{2}|^p \left|\frac{f'(t)}{f(t)}\right|^p f^{n-1}(t) \, dt \right)^{\frac{1}{p}}}{\left(\int_{ [2T, 3T]} f^{n-1}(t) \, dt\right)^{\frac{1}{p}}}
+
\frac{2 \left(\int_{[T, 4T]} f^{n-1}(t) \, dt \right)^{\frac{1}{p}}
}{T \left(\int_{ [2T, 3T]} f^{n-1}(t) \, dt\right)^{\frac{1}{p}}}.
\end{align*}
where we used that $\int_{N\times [a,b]} |\phi|^p F(t)\, dt \, dv_N = \Vert \phi\Vert_{L^p(N)}^p \int_a^b F(t)\, dt$. \medskip

For $p=1$ and using that $f$ is nondecreasing and positive this gives
 \begin{align*}
\frac{\Vert (\Dc^M -\lambda) \psi_T\Vert_{L^1}}{\Vert \psi_T\Vert_{L^1}} \leq
\frac{ f^{n-1}(4T)-f^{n-1}(T)}{2T f^{n-1}(2T)}
+
\frac{6f^{n-1}(4T)}{T f^{n-1}(2T)}\leq \frac{7f^{n-1}(4T)}{T f^{n-1}(2T)}.
\end{align*}
Then, \eqref{grow1} implies  that $\mathbb R\subset \sigma(\Dc_1^M)$ and hence, $ [0, \infty)\subset \sigma((\Dc_1^M)^2)$. Since the spectrum of $\Dc^2_{p}$ is $p$-independent for $1\leq p<\infty$ by Theorem~\ref{thmD2Lp}, we have $ [0, \infty)\subset \sigma(\Dc_2^2)$. Since $[0,\infty)$ is the maximal possible spectrum of the nonnegative self-adjoint operator $\Dc_2^2$, we have equality. \medskip

By Lemma B.10 and B.11 of \cite{AG16} if $n$ is not congruent $3$ modulo $4$ or $M$ admits an  isometry reversing  orientation that lifts to the spin structure, the $L^2$-spectrum of $\Dc$ is automatically symmetric, which implies $\sigma(\Dc_2)=\mathbb R$.  Moreover, if we also know that $\sigma(\Dc_1)\neq \mathbb{C},$ then  $\sigma(\Dc_p)=\mathbb{R}$ for all $p\in[1,\infty)$, by Corollary \ref{corDLp}.
\end{proof}

\subsection{Manifolds with asymptotically \texorpdfstring{$D^2$}{D2}-harmonic spinors}
As another application we consider manifolds with asymptotically $D^2$-harmonic spinors. This assumption is a relatively general one. In the following section we will consider the particular case of asymptotically flat manifolds which admit a Witten spinor.

\begin{thm} \label{thmWLp} Let $(M,g)$ be a complete Riemannian manifold with Clifford bundle $S$ and associated Dirac operator $D$.  Suppose that the Clifford contraction $\mathcal{R}$ in \eqref{eq:Wei} is bounded below and that the Ricci curvature of $M$ satisfies
\[
\text{Ric}(x) \geq -\delta(n) r^{-2}(x)
\]
where $r(x)=d(x_o,x)$ is the distance from $x$ to a fixed point $x_o$ and $\delta(n)$ is a small constant that depends only on the dimension of the manifold.

Moreover, assume that for all $R>0$  large enough there is a smooth section $\phi_R$ of $S$ restricted to  $M\setminus B_{x_o}(R)$ with $0<c_1<|\phi_R|< c_2$ on $M\setminus B_{x_o}(R)$, $|\nabla \phi_R|\leq C$, and $\|D^2\phi_R\|_{L^\infty}\to 0$ as $R\to \infty$. \medskip

Then, the $L^p$-spectrum of $D^2_{p}$ is $[0,\infty)$ for all $p\in[1,\infty]$.

If additionally, $\sigma (D_1)\neq \mathbb{C}$ and $\sigma(D_2)$ is symmetric, then  $\sigma(D_p)=\mathbb{R}$ for all $p\in[1,\infty)$.
\end{thm}

\begin{proof}
As Wang proves in \cite{Wang97}, under the above assumption on the Ricci curvature, the manifold has uniformly subexponential volume growth, and as a result the spectrum of $D_p^2$ is $L^p$ independent by Theorem \ref{thmD2Lp}. It therefore suffices to prove that the spectrum of $D_1^2$ contains $[0,\infty)$.

Let $\lambda \in [0,\infty)$.   By the estimates of \cite{Wang97}*{p. 478} it follows that for any $R$ large enough there is a function $\eta_R \in C^\infty(M)$ with support on $M\setminus B_{x_o}(R)$, such that
\[
\Vert
 \Delta \eta_R -\lambda \eta_R \Vert_{L^1}\leq (c/R) \; \Vert \eta_R \Vert_{L^1} \quad \text{and} \quad  \Vert \nabla \eta_R \Vert_{L^1}\leq (c/R) \; \Vert \eta_R \Vert_{L^1}.
 \]
where $c$ is a uniform positive constant. Since
\begin{equation} \label{e6_1}
 \begin{split}
  (D^2-\lambda)(\eta_R \phi_R)=& D(\eta_R D \phi_R  +  \nabla \eta_R \cdot \phi_R) - \lambda \eta_R \phi_R\\
  =& \eta_R D^2 \phi_R - 2\nabla_{ \nabla  \eta_R} \phi_R + (\Delta \eta_R -\lambda \eta_R) \phi_R
 \end{split}
 \end{equation}
we have the estimate
\begin{align*}
\frac{ \Vert (D^2-\lambda)(\eta_R \phi_R)\Vert_{L^1}}{\Vert \eta_R \phi_R\Vert_{L^1}}\leq&  \frac{\Vert \eta_R D^2 \phi_R\Vert_{L^1} + 2 \Vert \nabla \phi_R \Vert_{L^\infty} \, \Vert \nabla \eta_R \Vert_{L^1} + c_2 \Vert  (\Delta \eta_R -\lambda \eta_R)\Vert_{L^1}}{c_1\Vert \eta_R\Vert_{L^1}}\\
\leq & \frac{ C \left( 1/R +  \|D^2\phi_R\|_{L^\infty} \right) \Vert  \eta_R\Vert_{L^1}}{\Vert \eta_R\Vert_{L^1}}\to 0 \ \ \text{  as } \ \ R\to \infty.
\end{align*}

Under our additional assumption that $\sigma (D_1)\neq \mathbb{C}$ and $\sigma(D_2)$ is symmetric, we have that $\sigma(D_p)$ is independent of $p$, see Theorem~\ref{thmD2Lp}. Since $\sigma(D_2)$ is symmetric, then  $\sigma(D_2)=\mathbb{R}$. Hence  $\sigma(D_p)=\mathbb{R}$ for all $p$.
\end{proof}

\begin{corl}
Let $(M^n,g)$ be a complete  Riemannian spin manifold  with associated classical Dirac operator $\Dc$.  Suppose that the Ricci curvature of $M$ is asymptotically nonnegative as in Theorem \ref{thmWLp} and for all $R>0$  large enough there is a spinor $\phi_R$  on $M\setminus B_p(R)$ with $0<c_1<|\phi_R|< c_2$ on $M\setminus B_p(R)$, $|\nabla \phi_R|\leq C$, and $\|\Dc^2 \phi_R\|_{L^\infty}\to 0$ as $R\to \infty$.

Then, the $L^p$-spectrum of the classical square Dirac operator $\Dc^2_p$ is $[0,\infty)$ for all $p\in[1,\infty]$.

If in addition $n\neq 3$ modulo $4$ or  $M$ has an orientation-reversing isometry that lifts to spin structure, then  $\sigma(D_2)=\mathbb{R}$. If we also know  that $\sigma(\Dc_1)\neq \mathbb{C},$ then  $\sigma(\Dc_p)=\mathbb{R}$ for all $p\in[1,\infty)$.
\end{corl}

% Note that instead of $\|\Dc^2 \phi_R\|_{L^\infty}\to 0$ as $R\to \infty$ we could by the Schrödinger-Lichnerowicz formula have alternatively assumed that $\|\n^*\n \phi_R\|_{L^\infty}\to 0$ as $R\to \infty$ since the scalar curvature
\begin{proof}
The proof of the Corollary is immediately given by the last Theorem since in this case the scalar curvature is also asymptotically nonnegative and hence bounded from below.

The remaining statement for $\sigma(D_p)$ follows from Corollary~\ref{corDLp}.
\end{proof}

\section{The \texorpdfstring{$L^2$}{L2}-spectrum of the square Dirac operator} \label{S7}

In \cite{ChLu2}*{Theorem 2.3} the first author and Lu, proved a generalized Weyl criterion for the $L^2$-spectrum of a densely defined self-adjoint operator. This criterion can also be applied to the square of the Dirac operator over a Clifford bundle, since it is a self-adjoint and nonnegative operator. In fact, given that the resolvent operator $(D^2+1)^{-1}$ of $D^2$ is bounded on $L^\infty$, the proof of Theorem 1.1 in \cite{ChLu2} which was stated for the particular case of the Laplacian on functions directly generalizes to our setting, giving us the following result.
\begin{thm} \label{thm71} Let $(M,g)$ be a complete Riemannian manifold with Clifford bundle $S$ and associated Dirac operator $D$. Assume that for $\lambda>0$ there exists a sequence of sections of the Clifford bundle $\psi_i\in L^1\cap L^\infty$ in the domain of $D_1^2$ such that
\[
\frac{\|\psi_i\|_{L^\infty} \|(D^2-\lambda)\psi_i\|_{L^1}}{\|\psi_i\|_{L^2}^2} \to 0 \ \ \text{as} \ \ i \to \infty.
\]
Then $\lambda \in \sigma(D_2^2)$.

In addition, if for any compact subset $K$ of $M$ the support of the $\psi_i$ lies outside  $K$ for $i$ large enough, then $\lambda$ belongs to the essential spectrum of $D_2^2$.
\end{thm}

In Section 6 of \cite{ChLu2}  the first author and Lu  proved the following
\begin{lem} \label{lemEst}
Let $r(x)$ be the radial function with respect to $p$ on the manifold $M$.  Suppose that the Ricci curvature in the radial direction away from $p$ satisfies
\[
{\rm Ric}(\p_r, \p_r)\geq -(n-1) \delta(r),
\]
where  $\delta(r)$ is a continuous function on $\mathbb{R}^+$ such that ${\displaystyle \lim_{r\to\infty} \delta(r)=0}$, and $\p_r$ is the unit vector field in the radial direction. In the case that the volume of $M$ is finite, we make the further assumption that its volume does not decay exponentially at $p$. Then there exists a sequence of uniformly bounded test functions $\eta_i$ with disjoint supports such that
\[
\frac{\|(\Delta -\lambda)\eta_i\|_{L^1}}{\|\eta_i\|_{L^2}^2}\to 0 \ \ \text{as} \ \ i\to\infty
\]
\end{lem}

The functions $\eta_i$ are defined as
\begin{equation}
 \eta_i(x)=\chi_i({\tilde{r}}/{R_i})\, e^{\sqrt{-1}\sqrt{\lambda}\,\tilde{r}}.
\end{equation}
where $\tilde r(x)$ is a smoothing of the radial function $r(x)$, with $\nabla \tilde{r}$ uniformly bounded, and $\chi_i\colon  \mathbb{R}^+\to \mathbb{R}^+$ are smooth cut-off functions with support on $[x_i/R_i-1, y_i/R_i+1]$, where $x_i, y_i, R_i$ are large positive numbers such that $x_i>2R_i$ and $y_i>x_i+2R_i$.  Moreover,  $0\leq \chi_i \leq 1$ on $M$,  $\chi_i=1$ on $[x_i/R_i,y_i/R_i]$, and  $|\chi'_i|, |\chi''_i|$  are bounded uniformly.

Take a  family of sections of the Clifford bundle $\phi_i$  such that for all $i$ large enough $0<c_1<|\phi_i|<c_2$, and set $\psi_i = \eta_i \phi_i$. Then, by using \eref{e6_1},
\begin{align*}
 \|(D^2-\lambda)\psi_i\|_{L^1} \leq  &C \|D^2\phi_i\|_{L^1(B_p(y_i+R_i)\setminus B_p(x_i-R_i))} \\
 & \quad + \frac{C}{R_i} \|\nabla \phi_i\|_{L^1(B_p(y_i+R_i)\setminus B_p(x_i-R_i))} + \|(\Delta-\lambda)\eta_i\|_{L^1}
\end{align*}
Note that for manifolds as in Lemma \ref{lemEst} it was proved in  \cite{ChLu2} that $\text{Vol}(B_p(y_i+R_i)) \leq C \|\eta_i\|_{L^2}^2$, therefore if we have sections with the property that $D^2\phi_i \to 0$ as $i \to \infty$ and $|\nabla \phi_i|\leq C$, then
\begin{align*}
\frac{\|\psi_i\|_{L^\infty} \|(D^2-\lambda)\psi_i\|_{L^1}}{\|\psi_i\|_{L^2}^2}  \leq C \|D^2\phi_i\|_{L^\infty(B_p(y_i+R_i)\setminus B_p(x_i-R_i))} + \frac{C}{R_i} + \frac{\|(\Delta -\lambda)\eta_i\|_{L^1}}{\|\eta_i\|_{L^2}^2}.
\end{align*}
This observation together with Theorem \ref{thm71} proves the following result.
\begin{thm} \label{thm72}
Suppose that $M$ is a complete Riemannian manifold with Ricci curvature asymptotically nonnegative in the radial direction and satisfying the assumptions of Lemma \ref{lemEst}. In addition suppose that there exists a family of sections of the Clifford bundle $\phi_i$ such that (except possibly on a compact set $K$) for all $i$ large enough $0<c_1<|\phi_i|<c_2$, $|\nabla \phi_i|\leq C$, and $\|D^2\phi_i\|_{L^\infty} \to 0$ as $i \to \infty$.  Then $\sigma(D_2^2)=[0,\infty)$.
\end{thm}

The above result allows us to find a large class of manifolds over which the $L^2$ spectrum of $D^2$ is maximal. For the particular case of the square of the classical Dirac operator over an asymptotically Euclidean spin manifold we can use the Witten spinor constructed by Bartnik to compute the $L^2$-spectrum of $\Dc^2$ and $\Dc$~\cites{Bart, Witten81}.

\begin{Def}
A manifold $(M^n,g)$ with $n\geq 3$ is asymptotically flat if  its scalar curvature is in $L^1(M)$ and there is a compact set $K\subset M$ and a diffeomorphism
$\Phi\colon M \setminus K \to \mathbb{R}^n \setminus B_{r_o}(0)$ such that
\[
(\Phi_* g)_{ij} = \delta_{ij} + O(r^{2-n}), \quad \partial_k (\Phi_* g)_{ij} = O(r^{1-n}), \quad  \partial_{kl} (\Phi_* g)_{ij} = O(r^{-n}).
\]
\end{Def}
If in addition to being asymptotically flat the  manifold has positive scalar curvature and is spin, then Bartnik (see also Finster and Kath~\cite{FK}) proves the following result: We consider the boundary value problem
\begin{equation}\label{WitSp}
\Dc\psi =0, \ \ \lim_{|x|\to \infty} \psi(x) =\psi_0
\end{equation}
where $\psi_0$ is a constant spinor (that is viewed as a spinor on $M\setminus K$ using $\Phi$)
. As Bartnik proves, this problem has a unique solution, called a Witten spinor, which is smooth and satisfies the decay estimates
\[
\psi =\psi_0 + O(r^{2-n}), \quad \partial_k \psi =  O(r^{1-n}), \quad  \partial_{kl} \psi =  O(r^{-n}).
\]

By Theorem \ref{thm72} the following result is immediate.
\begin{prop}
Suppose that $M^n$ with $n\geq 3$ is an asymptotically flat spin manifold with positive scalar curvature. Then the  spectrum of $\Dc_2^2$ is $[0,\infty)$.
If additionally $n\neq 3$ modulo $4$ or  $M$ has an orientation-reversing isometry that lifts to spin structure, then $\sigma(\Dc_2)=\mathbb{R}$.

\end{prop}

Note that a Witten spinor $\psi$ which satisfies \eqref{WitSp} has the property that $\Dc^2\psi=0$. Hence,  Theorem \ref{thmD2Lp} leads to the following immediate result.
\begin{prop}
Let $M^n$ with $n\geq 3$ be a spin manifold whose Ricci curvature satisfies
\[
\text{Ric}(x) \geq -\delta(n) r^{-2}(x)
\]
where $r(x)=d(x_o,x)$ is the distance from $x$ to a fixed point $x_o$ and $\delta(n)$ is a small constant that depends only on the dimension of the manifold.

Suppose that $M$ admits a Witten spinor which is asymptotically constant and whose gradient vanishes asymptotically.

Then, the $\sigma(\Dc_p^2)=[0,\infty)$ for all $p\in[1,\infty]$.

If additionally $n\neq 3$ modulo $4$ or  $M$ has an orientation-reversing isometry that lifts to spin structure, then $\sigma(\Dc_2)=\mathbb{R}$. If we also know  that $\sigma(\Dc_1)\neq \mathbb{C},$ then  $\sigma(\Dc_p)=\mathbb{R}$ for all $p\in[1,\infty)$.
\end{prop}

\appendix
\section{On the definition of the Dirac operator on \texorpdfstring{$L^p$}{Lp}}\label{sec_def_D}

In \cite{AG16} the classical Dirac operator  on $L^p$, $\Dc_p$, for $1\leq p<\infty$ is defined as the completion of $\Dc\colon C_c^\infty(S)\subset L^p(S)\to C_c^\infty(S)\subset L^p(S)$ with respect to the graph norm. Analogously $\Dc_{p}^2 \left(=(\Dc_{p})^2\right)$ is defined and we use this definition as well for general Dirac-type operators as in Section~\ref{sec2}. However, the $L^p$-independence result we proved in Theorem~\ref{thmD2Lp} is a statement on the generator, $H_p$,  of the semigroup $e^{-tD^2}$ on $L^p$ as defined in Lemma~\ref{Dom3}.\medskip

We show next that for a Riemannian manifold with  Ricci curvature bounded from below this generator does in fact coincide with the operator $D^2$ defined above. We begin with the following result.

 \begin{lem} \label{lem8} Let $S_t\colon L^2\to L^2$ be a strongly continuous semigroup, and let $\mathfrak{D}\subset L^2\cap L^p$  be a core of $S_t$. For each $1\leq p<\infty$ we assume that $\mathfrak{D}$ is dense in $L^p$. Define $S_t\colon L^p\to L^p$ to be the completion of $S_t|_{\mathfrak{D}}$ in the graph norm.  Let $H_p\colon L^p \to L^p$ be the corresponding generator, in other words its domain is given by $\text{dom}(H_p)\define \{ \psi\in L^p\ |\ \exists \phi\in L^p\colon \Vert t^{-1}(\text{Id}-S_t) \psi - \phi\Vert_{L^p}\to 0 \text{ as }t\to 0\}$ and $H_p\psi\define \phi$. It follows that $S_t =e^{-H_pt}$.

Assume that $S_t(\mathfrak{D})\cup \mathfrak{D}\subset \text{dom}(H_p)$, that $S_tH_p=H_pS_t$ on $S_t (\mathfrak{D})\cup \mathfrak{D}$  and that  there is a positive constant $K>0$ such that $\|S_t \|_{p \to p}\leq e^{K t}$ for all $t>0$.
Then $\mathfrak{D}$ is a core for $H_p\colon L^p\to L^p$ and $H_p$ is the closure of $H_p|_{\mathfrak{D}}$ with respect to the graph norm.
 \end{lem}

\begin{proof}
 The proof goes along the lines of \cite{Rob}*{pp.~14--17}. However, although Robinson assumes that $S_t$ is a contraction semigroup and that $S_t(\mathfrak{D})\subset \mathfrak{D}$, we will prove that the condition $\|S_t \|_{p \to p}\leq e^{K t}$ is enough to obtain the statement, and that we can replace the assumption $S_t(\mathfrak{D})\subset \mathfrak{D}$ by $S_t(\mathfrak{D})\cup \mathfrak{D}\subset \text{dom}(H_p)$. \medskip

Fix a constant $\lambda < -K <0$ and define
  \[ R_\lambda (H_p)\phi \define -\int_0^\infty ds \;e^{\lambda s} S_s \phi\]
This operator is bounded on $L^p$ since for $\phi\in L^p$
\[ \Vert R_\lambda (H_p)\phi \Vert_{L^p} \leq \int_0^\infty ds \; e^{(\lambda +K)s} \Vert \phi\Vert_{L^p} = -(\lambda+K)^{-1} \Vert \phi\Vert_{L^p}.\]

 Moreover,
 \begin{align*}
  \Vert t^{-1}(\text{Id}-S_t)& R_\lambda(H_p)\phi + \phi - \lambda R_\lambda(H_p)\phi\Vert_{L^p}\\
   \leq& \; \Vert   t^{-1}\int_0^\infty ds \; e^{\lambda s} (S_{s+t}-S_{s}) \phi +\phi + \lambda \int_0^\infty ds \; e^{\lambda s} S_s \phi  \Vert_{L^p} \\
 \leq& \; \Vert  \left(t^{-1}(e^{-\lambda t}-1)  + \lambda\right) \int_0^\infty ds \; e^{\lambda s} S_{s} \phi +\phi - t^{-1} \int_0^t ds \; e^{\lambda (s-t)} S_s \phi \Vert_{L^p} \\
 \leq& \; | t^{-1}(e^{ -\lambda t}-1) + \lambda|  \; \Vert R_\lambda(H_p)\phi\Vert_{L^p} +\Vert   \phi - t^{-1} \int_0^t ds \; e^{\lambda (s-t)} S_s \phi \Vert_{L^p} \\
 & \ \to 0 \ \ \text{as} \ \ t\to 0
 \end{align*}

 This implies  $R_\lambda(H_p)(L^p)\subset \text{dom}(H_p)$ and $(\lambda \, \text{Id} -H_p)R_\lambda (H_p)\phi=\phi$. In particular, we have that $\text{Image}(\lambda \, \text{Id}-{H_p})=L^p$ and
 that $(\lambda \text{Id}-{H_p})$ is injective.\medskip

 Let $\hat{H_p}$ be the closure of $H_p|_{\mathfrak{D}}\colon L^p\to L^p$ in the graph norm.  If we can prove that $H_p=\hat{H}_p$, then $\mathfrak{D}$  would be a core of $H_p$. Given that $\text{Image}(\lambda \text{Id}-{H_p})=L^p$  and the map $(\lambda \text{Id}-{H_p})$ is injective for some $\lambda<-K<0$, if we can also show that $\text{Image}(\lambda \text{Id}-\hat{H_p})=L^p$, then we would immediately get that $H_p=\hat{H}_p$.\medskip

 Hence, it remains to prove that $\text{Image}(\lambda \text{Id}-\hat{H}_p)=L^p$ for some $\lambda<-K$. To this end we define
 \[ \psi_{n,k}\define -\frac{k}{n}{ \sum_{i=1}^{n} e^{\lambda \frac{ki}{n}} S_{\frac{ki}{n}}} \psi\]
 for any $\psi\in L^2\cap L^p$.  Then $\psi_k\define \lim_{n\to \infty}\psi_{n,k}= -\int_0^k ds \; e^{\lambda s} S_s \psi$ and $\lim_{k\to \infty} \psi_k= R_\lambda(H_p)\psi$ in $L^p$.
 Moreover, since $S_tH_p=H_pS_t$ on $S_t\mathfrak{D}\cup \mathfrak{D}$ we have  $\lim_{k\to \infty}\lim_{n\to \infty} (\lambda -H_p)\phi_{n,k}=\phi$ in $L^p$ for all $\phi \in S_t(\mathfrak{D})\cup \mathfrak{D}$. \medskip

Now let   $\phi\in \mathfrak{D}$. Then $\phi_{n,k} \in S_t(\mathfrak{D})\cup \mathfrak{D}$ and as a result, $(\lambda-H_p) \phi_{n,k}= ((\lambda-H_p)\phi)_{n,k}$.  Hence  $\lim_{k\to \infty} \lim_{n\to \infty} ((\lambda -H_p) \phi)_{n,k} =\phi$ in $L^p$ which implies $\mathfrak{D}\subset \text{Image}(\lambda -\hat{H}_p)\subset \text{Image}(\lambda -{H}_p)=L^p$. \medskip

Therefore, it remains to show that $\text{Image}(\lambda - \hat{H}_p)$ is norm closed, which then implies  $\text{Image}(\lambda\text{Id}-\hat{H}_p) = L^p$:
Since $(\lambda\text{Id}- \hat{H}_p)^{-1}$ is bounded, then $(\lambda\text{Id}- \hat{H}_p)$ is bounded below. To show that $\text{Image}(\lambda\text{Id}-\hat{H}_p) = L^p$, let $\phi_i\in \text{Image}(\lambda - \hat{H}_p)$ and $\phi_i\to \phi$ in $L^p$. Since $(\lambda -H_p)^{-1}$ is bounded and $\text{dom}(H_p) \subset \text{dom}(\hat{H}_p)$, $\psi_i\define (\lambda -\hat{H}_p)^{-1}\phi_i=(\lambda - {H}_p)^{-1}\phi_i\in \text{dom}(\hat{H}_p)$ converges to $\psi\define (\lambda-H_p)^{-1}\phi$ in $L^p$. Thus, $\psi_i\to \psi$ and $\hat{H}_p\psi_i\to \lambda\psi -\phi$  in $L^p$. Since $\hat{H}_p$ is closed in graph norm $\psi\in \text{dom}(\hat{H}_p)$ and hence $\text{Image}(\lambda - \hat{H}_p)$ is norm closed
 \end{proof} \medskip

 \begin{corl}\label{cor_dom_D^2}
  Let $(M,g)$ be a complete Riemannian manifold with Clifford bundle $S$ and associated Dirac operator  $D$. Suppose that the Ricci curvature of $(M,g)$ and the respective Clifford contraction are bounded from below.   Then, for $1\leq p<\infty$ the generator of the semigroup $e^{-tD^2}\colon L^p\to L^p$ is the square of the completion of  $D_c\colon C_c^\infty(S)\subset L^p(S)\to C_c^\infty(S)\subset L^p(S)$ with respect to the graph norm.
 \end{corl}

\begin{proof}
 Set $S_t\define e^{-tD^2}\colon L^2\to L^2$. Then $S_t$ is a strongly continuous semigroup  with $|S_t|\leq e^{K_1 t}$ by Lemma~\ref{Dom3}.  Let $\mathfrak{D}=C_c^\infty(S)$. Then,  $\mathfrak{D}$ is a core of $S_t$ and dense in $L^p$ for $1\leq p<\infty$.  Let  $S_t\colon L^p\to L^p$ be the completion of $S_t|_{\mathfrak{D}}$ in the graph norm. The corresponding generator of this semigroup will be denoted by $H_p$.  Clearly, $\mathfrak{D} \subset \text{dom}(H_p)$.\medskip

Let $\psi\in \mathfrak{D}$ and $H_p\psi=\phi$. Then,
\[
\Vert t^{-1}(\text{Id}-S_t) S_t\psi- S_t\phi\Vert_{L^p}\leq e^{K_1t} \Vert t^{-1}(\text{Id}-S_t) \psi- \phi\Vert_{L^p}\to 0 \quad \text{as} \quad t\to 0.
\]
Hence $ S_t(\mathfrak{D}) \subset \text{dom}(H)$. Moreover,  by construction  $H_2$ is the generator $D^2$ of the semigroup $S_t$ on $L^2$, and $S_t$ and $H_2$ commute on $\text{dom}(H_2)\subset L^2$ by functional calculus.  Since all $H_p$ coincide on $\text{dom}(H_2)\cap \text{dom}(H_p)$ (and the same for $S_t$ for the different $p$) we have that $S_t$ and $H_p$ commute on $S_t(\mathfrak{D})\cup \mathfrak{D}$.\medskip

Using Lemma \ref{lem8} and the fact that $H|_{\mathfrak{D}}=D_c^2$ we get that  $H_p$ is the completion of $D_c^2\colon C_c^\infty(S)\subset L^p(S)\to C_c^\infty(S)\subset L^p(S)$ with respect to the graph norm.
\end{proof}

\section{Spectrum for bounded geometry}

In \cite{AG16}*{Lem~B.8} the implication: $\lambda^2\in \sigma (D_p^2)$ if and only if  $\lambda$ or $-\lambda$ belongs to $\sigma (D_p)$ is proved for the spin Dirac operator under the additional assumption that $\sigma (D_p)\neq \mathbb C$. The proof immediately carries over to manifolds with Clifford bundles and associated Dirac operator $D$. Under additional geometric assumptions on the manifold and the Clifford bundle, the assumption
$\sigma (D_p)\neq \mathbb C$ is automatic. We add here a corresponding statement for completeness.

\begin{lem}\label{app_b}
 Let $(M,g)$ be a complete Riemannian manifold with Clifford bundle $S$ and associated Dirac operator $D$. Assume that $(M,g)$ and $S$ are of bounded geometry. Let $p\in [1,\infty]$. Then $\lambda^2\in \sigma (D_p^2)$ if and only if  $\lambda$ or $-\lambda$ belongs to $\sigma (D_p)$.
\end{lem}

\begin{proof}
 The `only if'-direction holds without the assumption of bounded geometry as shown in \cite{AG16}*{Lem~B.8}. The proof of the `if'-direction is also the same as therein as soon as we have: If $\psi\in L^p$ and $D^2\psi\in L^p$, then $D\psi\in L^p$. In case of bounded geometry this follows from elliptic regularity theory. We have $\Vert D\phi\Vert_{L^p}\leq c\left( \Vert \phi\Vert_{L^p}+ \Vert D^2\phi\Vert_{L^p}\right)$ for some $c>0$ and hence $D\psi\in L^p$.
\end{proof}

We note that in the case of a spin manifold $(M,g)$ with $S$ the corresponding spinor bundle, the assumption that $(M,g)$ is of bounded geometry automatically implies that $S$ is of bounded geometry.

% \listofchanges[style=summary]


\begin{bibdiv}
\begin{biblist}


\bib{AG16}{article}{
   author={Ammann, Bernd},
   author={Gro\ss e, Nadine},
   title={$L^p$-spectrum of the Dirac operator on products with hyperbolic
   spaces},
   journal={Calc. Var. Partial Differential Equations},
   volume={55},
   date={2016},
   number={5},
   pages={Art. 127, 36},
   issn={0944-2669},
}

\bib{Aut}{article}{
   author={Auterhoff, J\"{u}rgen},
   title={Interpolationseigenschaften des Spektrums linearer Operatoren auf
   $L^{p}$-R\"{a}umen},
   language={German},
   journal={Math. Z.},
   volume={184},
   date={1983},
   number={3},
   pages={397--406},
   issn={0025-5874},
}

 \bib{Barhyp}{article}{
  author = {B\"{a}r, Christian},
  title = {{The Dirac operator on hyperbolic manifolds of finite volume}},
%  FJournal = {{Journal of Differential Geometry}},
  journal = {{J. Differ. Geom.}},
  volume = {54},
  number = {3},
 pages = {439--488},
 date = {2000},
 }

\bib{BGM}{article}{
  author={B\"{a}r, Christian},
   author={Gauduchon, Paul},
   author={Moroianu, Andrei},
   title={Generalized cylinders in semi-{R}iemannian and {S}pin
              geometry},
   journal={Math. Z.},
   volume={249},
   date={2005},
   number={3},
   pages={545--580},
}

\bib{Bart}{article}{
   author={Bartnik, Robert},
   title={The mass of an asymptotically flat manifold},
   journal={Comm. Pure Appl. Math.},
   volume={39},
   date={1986},
   number={5},
   pages={661--693},
   issn={0010-3640},
}


\bib{BGV}{book}{
   author={Berline, Nicole},
   author={Getzler, Ezra},
   author={Vergne, Mich\`ele},
   title={Heat kernels and Dirac operators},
   series={Grundlehren Text Editions},
   note={Corrected reprint of the 1992 original},
   publisher={Springer-Verlag, Berlin},
   date={2004},
   pages={x+363},
   isbn={3-540-20062-2},
}

\bib{Char1}{article}{
   author={Charalambous, Nelia},
   title={On the $L^p$-independence of the spectrum of the Hodge
   Laplacian on non-compact manifolds},
   journal={J. Funct. Anal.},
   volume={224},
   date={2005},
   number={1},
   pages={22--48},
}

\bib{ChLu1}{article}{
   author={Charalambous, Nelia},
   author={Lu, Zhiqin},
   title={Heat kernel estimates and the essential spectrum on weighted
   manifolds},
   journal={J. Geom. Anal.},
   volume={25},
   date={2015},
   number={1},
   pages={536--563},
   issn={1050-6926},
}


\bib{ChLu2}{article}{
   author={Charalambous, Nelia},
   author={Lu, Zhiqin},
   title={On the spectrum of the Laplacian},
   journal={Math. Ann.},
   volume={59},
   date={2014},
   number={1-2},
   pages={211--238},
}

\bib{Chavel}{book}{
   author={Chavel, Isaac},
   title={Eigenvalues in Riemannian geometry},
   series={Pure and Applied Mathematics},
   volume={115},
   note={Including a chapter by Burton Randol;
   With an appendix by Jozef Dodziuk},
   publisher={Academic Press, Inc., Orlando, FL},
   date={1984},
   pages={xiv+362},
}

\bib{Davies}{book}{
   author={Davies, E. B.},
   title={Heat kernels and spectral theory},
   series={Cambridge Tracts in Mathematics},
   volume={92},
   publisher={Cambridge University Press},
   place={Cambridge},
   date={1990},
   pages={x+197},
}

\bib{Dav2}{article}{
   author={Davies, E. B.},
   title={Heat kernel bounds for second order elliptic operators on
   Riemannian manifolds},
   journal={Amer. J. Math.},
   volume={109},
   date={1987},
   number={3},
   pages={545--569},
}


\bib{davies97}{article}{
 Author = {E. B. {Davies}},
 Title = {{\(L^p\)-spectral theory of higher-order elliptic differential operators}},
 Journal = {{Bull. Lond. Math. Soc.}},
 ISSN = {0024-6093},
 Volume = {29},
 Number = {5},
 Pages = {513--546},
 Year = {1997},
}

\bib{DP}{article}{
   author={Dunford, Nelson},
   author={Pettis, B. J.},
   title={Linear operations on summable functions},
   journal={Trans. Amer. Math. Soc.},
   volume={47},
   date={1940},
   pages={323--392},
   issn={0002-9947},
}
	

\bib{EN}{book}{
   author={Engel, Klaus-Jochen},
   author={Nagel, Rainer},
   title={One-parameter semigroups for linear evolution equations},
   series={Graduate Texts in Mathematics},
   volume={194},
   note={With contributions by S. Brendle, M. Campiti, T. Hahn, G. Metafune,
   G. Nickel, D. Pallara, C. Perazzoli, A. Rhandi, S. Romanelli and R.
   Schnaubelt},
   publisher={Springer-Verlag, New York},
   date={2000},
   pages={xxii+586},
}
	

\bib{FK}{article}{
   author={Finster, Felix},
   author={Kath, Ines},
   title={Curvature estimates in asymptotically flat manifolds of positive
   scalar curvature},
   journal={Comm. Anal. Geom.},
   volume={10},
   date={2002},
   number={5},
   pages={1017--1031},
   issn={1019-8385},
}

 \bib{HemVo}{article}{
   author={Hempel, Rainer},
   author={Voigt, J{\"u}rgen},
   title={On the $L_p$-spectrum of Schr\"odinger operators},
   journal={J. Math. Anal. Appl.},
   volume={121},
   date={1987},
   number={1},
   pages={138--159},
}

\bib{HemVo2}{article}{
   author={Hempel, Rainer},
   author={Voigt, J{\"u}rgen},
   title={The spectrum of a Schr\"odinger operator in $L_p({\bf R}^\nu)$ is $p$-independent},
   journal={Comm. Math. Phys.},
   volume={104},
   date={1986},
   number={2},
   pages={243--250},
}
	

	

 \bib{HSU}{article}{
   author={Hess, H.},
   author={Schrader, R.},
   author={Uhlenbrock, D. A.},
   title={Domination of semigroups and generalization of Kato's inequality},
   journal={Duke Math. J.},
   volume={44},
   date={1977},
   number={4},
   pages={893--904},
}

\bib{Kato}{book}{
   author={Kato, Tosio},
   title={Perturbation theory for linear operators},
   series={Die Grundlehren der mathematischen Wissenschaften, Band 132},
   publisher={Springer-Verlag New York, Inc., New York},
   date={1966},
   pages={xix+592},
}

\bib{Kaw1}{article}{
   author={Kawai, Shigeo},
   title={Essential spectrums of the Dirac operators on non-compact
   positively curved manifolds},
   journal={Far East J. Math. Sci. (FJMS)},
   volume={30},
   date={2008},
   number={2},
   pages={309--324},
   issn={0972-0871},
   review={\MR{2477774}},
}


\bib{Kaw2}{article}{
   author={Kawai, Shigeo},
   title={Spectrum of the Dirac operator on manifold with asymptotically
   flat end},
   journal={J. Geom. Phys.},
   volume={110},
   date={2016},
   pages={195--212},
   issn={0393-0440},
   review={\MR{3566109}},
   doi={10.1016/j.geomphys.2016.07.008},
}
	

\bib{Kordu}{article}{
 Author = {Yu. A. {Kordyukov}},
 Title = {{\(L^ p\)-theory of elliptic differential operators on manifolds of bounded geometry}},
 Journal = {{Acta Appl. Math.}},
 Volume = {23},
 Number = {3},
 Pages = {223--260},
 Year = {1991},
 }

\bib{ReSiII}{book}{
   author={Reed, Michael},
   author={Simon, Barry},
   title={Methods of modern mathematical physics. II. Fourier analysis,
   self-adjointness},
   publisher={Academic Press [Harcourt Brace Jovanovich, Publishers], New
   York-London},
   date={1975},
   pages={xv+361},
}

\bib{Rob}{book}{
   author={Robinson, Derek W.},
   title={Basic theory of one-parameter semigroups.},
   publisher={Proceedings of the Centre for Mathematical Analysis, Australian National University, 2. Australian National University, Centre for Mathematical Analysis, Canberra},
   date={1982},
   pages={iv+138},
}

\bib{Roe}{book}{
   author={Roe, John},
   title={Elliptic operators, topology and asymptotic methods},
   series={Pitman Research Notes in Mathematics Series},
   volume={395},
   edition={2},
   publisher={Longman, Harlow},
   date={1998},
   pages={ii+209},
}

\bib{Ros}{article}{
   author={Rosenberg, Steven},
   title={Semigroup domination and vanishing theorems},
   conference={
      title={Geometry of random motion},
      address={Ithaca, N.Y.},
      date={1987},
   },
   book={
      series={Contemp. Math.},
      volume={73},
      publisher={Amer. Math. Soc., Providence, RI},
   },
   date={1988},
   pages={287--302},
}
		
	
\bib{saco}{article}{
   author={Saloff-Coste, Laurent},
   title={Uniformly elliptic operators on Riemannian manifolds},
   journal={J. Differential Geom.},
   volume={36},
   date={1992},
   number={2},
   pages={417--450},
}
	

\bib{sturm}{article}{
   author={Sturm, Karl-Theodor},
   title={On the $L^p$-spectrum of uniformly elliptic operators on
   Riemannian manifolds},
   journal={J. Funct. Anal.},
   volume={118},
   date={1993},
   number={2},
   pages={442--453},
}
	
\bib{Wang97}{article}{
   author={Wang, Jiaping},
   title={The spectrum of the Laplacian on a manifold of nonnegative Ricci
   curvature},
   journal={Math. Res. Lett.},
   volume={4},
   date={1997},
   number={4},
   pages={473--479},
}
	
\bib{Witten81}{article}{
 Author = {Edward {Witten}},
 Title = {{A new proof of the positive energy theorem.}},
 Journal = {{Commun. Math. Phys.}},
 ISSN = {0010-3616},
 Volume = {80},
 Pages = {381--402},
 Year = {1981}
}

			





\end{biblist}
\end{bibdiv}
\end{document}